\documentclass[12pt]{article}%
\usepackage{amsmath,amssymb}
\usepackage[dvips]{epsfig}
\usepackage{graphicx}
\usepackage{pst-plot, pstricks, pst-node}
\usepackage{amsmath}
\usepackage{amsfonts}
\usepackage{amssymb}
\usepackage{graphicx}%
\providecommand{\U}[1]{\protect\rule{.1in}{.1in}}
\newtheorem{theorem}{Theorem}

\newtheorem{definition}[theorem]{Definition}

\newtheorem{lemma}[theorem]{Lemma}

\newtheorem{proposition}[theorem]{Proposition}
\newtheorem{remark}[theorem]{Remark}

\newenvironment{proof}[1][Proof]{\noindent\textbf{#1.} }{\ \rule{0.5em}{0.5em}}
\begin{document}

\title{Standing waves in a counter-rotating vortex filament pair\thanks{This is a
corrected version of the printed article.}}
\author{Carlos Garc\'{\i}a-Azpeitia\thanks{Departamento de Matem\'{a}ticas, Facultad
de Ciencias, Universidad Nacional Aut\'{o}noma de M\'{e}xico, 04510 M\'{e}xico
DF, M\'{e}xico}}
\maketitle

\begin{abstract}
The distance among two counter-rotating vortex filaments satisfies a beam-type
of equation according to the model derived in \cite{Ma95}. This equation has
an explicit solution where two straight filaments travel with constant speed
at a constant distance. The boundary condition of the filaments is $2\pi
$-periodic. Using the distance of the filaments as bifurcating parameter, an
infinite number of branches of periodic standing waves bifurcate from this
initial configuration with constant rational frequency along each branch.



\end{abstract}

\section*{Introduction}

In \cite{Ma95} is derived a model for the movement of almost-parallel vortex
filaments from the three-dimensional Euler equation. This model takes in
consideration the interaction between different filaments and an approximation
for the self-induction of each filament. The paper \cite{Ma95} presents a
first analysis of the finite time collapse of two filaments with negative
circulations; close to collapse, the model of vortex filaments as an
approximation to the Euler equation loses validity. Later, \cite{Po03} proves
that two filaments with positive circulations, and also three filaments with
positive circulations near an equilateral triangle, evolve without collapse
for all time. On the other hand, the evidence shown in
\cite{BaMi,BaMi12,Ma95,Ne01} suggests that two filaments with opposite
circulations develop collapse for many initial configurations. Our aim is to
investigate the existence of nontrivial periodic solutions of two vortex
filaments with opposite circulations, which evolve without collapse and
remains valid within the hypothesis of the model for all time.

The counter-rotating filament pair consists of two filaments with opposite
circulations and same strength. In the model deduced in \cite{Ma95}, the
almost parallel filaments are parameterized by
\[
(u_{j}(t,s),s)\in\mathbb{C}\times\mathbb{R}~,\text{\qquad}j=1,2~,
\]
and the distance among the filaments $w_{1}=u_{1}-u_{2}$ satisfies the
beam-type of equation%
\begin{equation}
\partial_{t}^{2}w_{1}=-\partial_{s}^{4}w_{1}+\partial_{s}^{2}\left(
\left\vert w_{1}\right\vert ^{-2}w_{1}\right)  \text{.} \label{EQ}%
\end{equation}
This equation has the explicit solution $w_{1}(t,s)=a$ that corresponds to the
solution of two straight filaments traveling with speed $a^{-1}$ at distance
$a$. The aim is to construct $2\pi/\nu$-periodic families of standing wave
bifurcating from this initial configuration, where the filaments have $2\pi
$-periodic boundary condition.

The present paper adopts the strategy followed in \cite{Ki79} for the wave
equation, where bifurcation of periodic solutions is proven to exist using
external parameters such as the amplitude, while the frequency is a fixed
rational. In \cite{Ki00} and \cite{Ry01} this result was improved to obtain
global bifurcation of perio\-dic solutions in spherical domains. A main
difference with our result is that the equation is semilinear and requires
special estimates.

\begin{theorem}
\label{1}For each number $q$, there is an infinite number of non-resonant
(Definition \ref{NR}) amplitudes $a_{0}$'s given by
\begin{equation}
a_{0}^{-2}:=\frac{2}{q}-\frac{1}{k_{0}^{2}q^{2}} \label{a}%
\end{equation}
for some $k_{0}\in\mathbb{N}$. For each of these non-resonant $a_{0}$'s, there
is a local continuum of $2\pi q/p$-periodic solution bifurcating from the
straight filaments with distance $a_{0}$, where $p=qk_{0}^{2}-1$. The local
bifurcation consists of standing waves satisfying the symmetries%
\begin{align}
w_{1}(t,s)  &  =w_{1}(-t,s)=w_{1}(t,-s)=w_{1}(t,s+2\pi/k_{0})\nonumber\\
&  =\bar{w}_{1}\left(  t+l_{0}(q\pi/p),s\right)  \text{,} \label{sym}%
\end{align}
where $l_{0}=0$, and the estimate%
\begin{equation}
w_{1}(t,s)=a_{0}+i^{l_{0}}b\cos\left(  pt/q\right)  \cos k_{0}s+\mathcal{O}%
_{C^{4}}(b^{2})\text{,} \label{es}%
\end{equation}
where $b\in\lbrack0,b_{0}]$ gives a parameterization of the local bifurcation.
\end{theorem}

The symmetries imply that the standing waves are even in $t$ and even and
$2\pi/k_{0}$-periodic in $s$. Setting $w_{1}=x+iy$, for $l_{0}=0$, the
symmetry (\ref{sym}) implies that $y(t,s)=0$, i.e. the orbits of the standing
waves are orthogonal to the traveling direction of the filaments. While for
$l_{0}=1$, this symmetry implies that
\[
x(t,s)=x\left(  t+(q\pi/p),s\right)  ,\qquad y(t,s)=-y\left(  t+(q\pi
/p),s\right)  ,
\]
i.e. the orbits of the standing waves resemble eight
figures.\begin{figure}[ptbh]
\resizebox{12.0cm}{!}{
\begin{pspicture}(-2.5,-2)(2.5,2)\SpecialCoor
\pcline[offset=12pt]{[-]}(-2,-1.5)(2,-1.5)
\ncput*[nrot=:U]{$a_{0}$}
\psline[arrowscale=1]{*->}(0,0)(.5,.5)
\rput[b](.5,.5){$\small a^{-1}_{0}$}
\psline(-2,-1.5)(-2,1.5)
\psline[linestyle=dotted](-1.5,-1.5)(-1.5,1.5)
\psline[linestyle=dotted](-2.5,-1.5)(-2.5,1.5)
\psline[linestyle=dashed,arrowscale=1.5]{<-}(-1.5,1.5)(-2.5,1.5)
\psline[linestyle=dashed,arrowscale=1.5]{->}(-1.5,0)(-2.5,0)
\psline[linestyle=dashed,arrowscale=1.5]{<-}(-1.5,-1.5)(-2.5,-1.5)
\psecurve(-1.5,3)(-2.5,1.5)(-1.5,0)(-2.5,-1.5)(-1.5,-3)
\psline(2,-1.5)(2,1.5)
\psline[linestyle=dotted](1.5,-1.5)(1.5,1.5)
\psline[linestyle=dotted](2.5,-1.5)(2.5,1.5)
\psline[linestyle=dashed,arrowscale=1.5]{<-}(1.5,1.5)(2.5,1.5)
\psline[linestyle=dashed,arrowscale=1.5]{->}(1.5,0)(2.5,0)
\psline[linestyle=dashed,arrowscale=1.5]{<-}(1.5,-1.5)(2.5,-1.5)
\psecurve(1.5,3)(2.5,1.5)(1.5,0)(2.5,-1.5)(1.5,-3)
\NormalCoor\end{pspicture}
\qquad
\begin{pspicture}(-2.5,-2)(2.5,2)
\SpecialCoor
\pcline[offset=12pt]{[-]}(-2,-1.5)(2,-1.5)
\ncput*[nrot=:U]{$a_{0}$}
\psline[arrowscale=1]{*->}(0,0)(.5,.5)
\rput[b](.5,.5){$\small a^{-1}_{0}$}
\psline(-2,-1.5)(-2,1.5)
\psline[linestyle=dotted](-2.5,1)(-2.5,-2)
\psline[linestyle=dotted](-1.5,2)(-1.5,-1)
\psline[linestyle=dashed]{->}(-2.5,1)(-1.5,2)
\psline[linestyle=dashed]{->}(-1.5,.5)(-2.5,-.5)
\psline[linestyle=dashed]{->}(-2.5,-2)(-1.5,-1)
\psecurve(-1.5,3.5)(-2.5,1)(-1.5,.5)(-2.5,-2)(-1.5,-2.5)
\psline(2,-1.5)(2,1.5)
\psline[linestyle=dotted](1.5,1)(1.5,-2)
\psline[linestyle=dotted](2.5,2)(2.5,-1)
\psline[linestyle=dashed]{<-}(1.5,1)(2.5,2)
\psline[linestyle=dashed]{<-}(2.5,.5)(1.5,-.5)
\psline[linestyle=dashed]{<-}(1.5,-2)(2.5,-1)
\psecurve(1.5,2.5)(2.5,2)(1.5,-.5)(2.5,-1)(1.5,-3.5)
\NormalCoor
\end{pspicture}
}\caption{Illustration of the two kind of solutions bifurcating from the
straight counter-rotating vortex filaments, initially separated by $a_{0}$ and
traveling with speed $a_{0}^{-1}$. Left: case\ $l_{0}=0$. Rigth: case
$l_{0}=1$.}%
\end{figure}
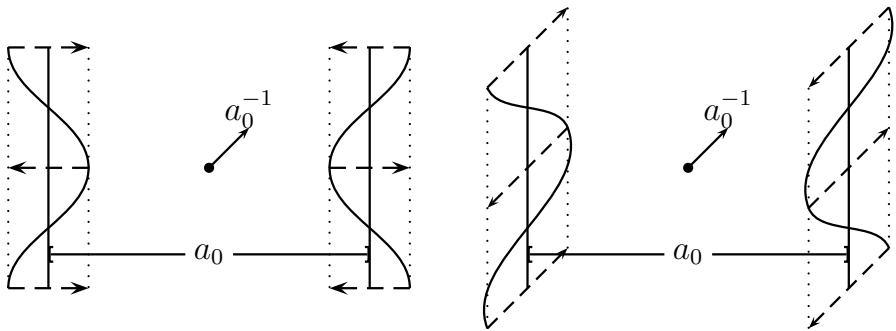

In \cite{GaCr15} the existence of standing waves for $n$ vortex filaments of
equal vorticities from a uniformly rotating central configuration is
investigated. In the case of two filaments, the distance $w_{1}(s,t)$
satisfies the Schr\"{o}dinger equation $\partial_{t}w_{1}=i\left(
\partial_{ss}w_{1}+\left\vert w_{1}\right\vert ^{-2}w_{1}\right)  $, which has
the explicit solution $ae^{ia^{-2}t}$ that corresponds to the solution where
the two filaments rotate with frequency $a^{-2}$ at distance $a$. This article
proves that the co-rotating filament pair has families of standing waves with
amplitudes varying over a Cantor set for irrational diophantine frequencies
$a^{-2}$. In order to solve the small divisor problem that appears due the
fact that the standing waves have irrational frequencies, \cite{GaCr15}
implements a Nash-Moser procedure. This result is different but complementary
to the existence of standing waves with rational frequencies in the
counter-rotating filament pair. Indeed, the method in \cite{GaCr15} can be
used to obtain standing waves with irrational frequencies in the
counter-rotating filament pair. The method presented here can be used to
obtain standing waves with rational frequencies in the co-rotating filament pair.

Nash-Moser methods for wave, Schr\"{o}dinger and beam equations have been
implemented in \cite{Bo95}, \cite{Be07}, \cite{CrWa93} and references therein.
Different methods which do not involve small divisor problems have been
developed to prove existence of periodic solutions. In these methods, the
frequency is fixed to a rational or a badly approximated irrational. For
rational frequencies, the linear operator has isolated point spectrum, but the
kernel associated to the bifurcation problem may have infinite dimension, see
\cite{AmZe80}, \cite{Ki79} and \cite{Ra78}. For strong irrational frequencies,
the inverse of the linear operator is bounded, but the inverse lacks
compactness, see \cite{Ba}, \cite{Be07} and \cite{Ra}. These methods have
limited applicability to semilinear beam equations \cite{Ki79,Ba}, which is
the case of our problem, and also in Schr\"{o}dinger equations, which is the
case of the co-rotating vortex filament pair. We recommend \cite{Be07} for an
overview of different applications of these methods to Hamiltonian PDEs.

The proof of our theorem relies on the fact that the inverse operator
associated to the bifurcation problem gains two spatial derivatives which
compensates the derivatives appearing in the nonlinearity. More precisely, the
bifurcation problem is equivalent to solve
\[
L(a)u+\partial_{s}^{2}g(u)=0,
\]
for a perturbation $u$, where $L$ is the linearized operator and $g(u)$ is an
analytic nonlinear operator. In the Fourier basis given by $e^{i\left(
jt+ks\right)  }$, the eigenvalues of $L$ are%
\[
\lambda_{j,k,l}(a)=\left(  pj/q\right)  ^{2}-k^{4}+(-1)^{l}a^{-2}k^{2}\text{
,}%
\]
for $\left(  j,k,l\right)  \in\mathbb{Z}^{2}\times\mathbb{Z}_{2}$. The
eigenvalues $\lambda_{\pm1,\pm k_{0},l_{0}}$ are zero at $a_{0}$ and the
others satisfy
\[
\lambda_{j,k,l}(a_{0})\in\left(  qk_{0}\right)  ^{-2}\mathbb{Z}\text{.}%
\]
Thus, the operator $L(a)$ can be inverted in the orthogonal complement of the
kernel for a neighborhood of $a_{0}$ . By choosing $a_{0}^{-2}\in\left(
0,2/q\right)  $, the projected inverse $\left(  PLP\right)  ^{-1}$ gains two
spatial derivatives due to the sharp estimate
\[
\lambda_{j,k,l}(a)\gtrsim k^{2}+\left\vert j\right\vert ~.
\]
However, the inverse $\left(  PLP\right)  ^{-1}$ does not gain extra
derivatives and $\partial_{s}^{2}\left(  PLP\right)  ^{-1}$ lacks the
necessary compactness to establish the global bifurcation by the classical
Rabinowitz approach.

The paper is structured as follows. In Section 1, we present the equation that
describes the dynamics of the distance of two straight vortex filaments. In
Section 2 the existence of standing waves is obtained by the Lyapunov-Schmidt
reduction method. In Section 3 the range equation is solved by the contracting
mapping theorem. In Section 4 the bifurcation equation is solved using the
symmetries of the problem and the Crandall-Rabinowitz theorem. Existence of
traveling waves solutions is discussed in Section 5.

\section{Setting the problem}

The counter-rotating filament pair consists of two filaments with circulations
$\Gamma_{1}=1$ and $\Gamma_{2}=-1$. According to \cite{Ma95}, the equations
that describe the dynamics of two almost parallel filaments are%
\begin{align*}
\partial_{t}u_{1}  &  =i\left(  \partial_{s}^{2}u_{1}-\frac{1}{2}\frac
{u_{1}-u_{2}}{\left\vert u_{1}-u_{2}\right\vert ^{2}}\right)  \text{,}\\
\partial_{t}u_{2}  &  =i\left(  -\partial_{s}^{2}u_{2}+\frac{1}{2}\frac
{u_{2}-u_{1}}{\left\vert u_{2}-u_{1}\right\vert ^{2}}\right)  \text{.}%
\end{align*}
The factor $1/2$ may be obtained by scaling the dimensions and is useful in
the discussion of our analysis.

The coordinates
\[
w_{1}=u_{1}-u_{2},\qquad w_{2}=u_{1}+u_{2},
\]
represent the distance and the center of mass of two filaments, respectively.
In these coordinates, the equations are
\[
\partial_{t}w_{1}=i\partial_{s}^{2}w_{2},\qquad\partial_{t}w_{2}=i\left(
\partial_{s}^{2}w_{1}-\left\vert w_{1}\right\vert ^{-2}w_{1}\right)  \text{.}%
\]
Therefore, the distance $w_{1}$ satisfies the equation%
\begin{equation}
\partial_{t}^{2}w_{1}=i\partial_{s}^{2}\partial_{t}w_{2}=-\partial_{s}%
^{4}w_{1}+\partial_{s}^{2}(\left\vert w_{1}\right\vert ^{-2}w_{1})\text{,}
\label{v}%
\end{equation}
and the center of mass $w_{2}$ can be obtained from $w_{1}$ by integration:%
\begin{equation}
w_{2}(t,s)=i\int_{0}^{t}\left(  \partial_{s}^{2}w_{1}-\left\vert
w_{1}\right\vert ^{-2}w_{1}\right)  dt+w_{2}(0,s)\text{.} \label{w}%
\end{equation}

The explicit solution%
\[
w_{1}(t,s)=a\qquad w_{2}(t,s)=-ia^{-1}t\text{,}%
\]
corresponds to the solution where the filaments travel with constant speed. We
look for bifurcation of solution from this initial configuration of the form
\[
w_{1}(t,s)=a(1-u(\nu t,s)),
\]
where $u$ is $2\pi$-periodic in $t$ and $s$.

The equation that satisfies the perturbation $u$ is%
\[
\nu^{2}\partial_{t}^{2}u=-\partial_{s}^{4}u+\frac{1}{a^{2}}\partial_{s}%
^{2}\left(  \frac{1}{1-\bar{u}}\right)  \text{.}%
\]
Using a Taylor expansion, this equation is equivalent to%
\begin{equation}
\nu^{2}\partial_{t}^{2}u=-\partial_{s}^{4}u+a^{-2}\partial_{s}^{2}\bar
{u}+\partial_{s}^{2}g(\bar{u})\text{,}%
\end{equation}
where%

\begin{equation}
g(\bar{u})=a^{-2}\frac{\bar{u}^{2}}{1-\bar{u}}=a^{-2}\sum_{j=2}^{\infty}%
\bar{u}^{j}%
\end{equation}
is analytic for $\left\vert u\right\vert <1$.

\section{The Lyapunov-Schmidt reduction}

Hereafter the frequency $\nu$ is fixed to the rational%
\[
\nu=\frac{p}{q},
\]
where $p$ and $q$ are relative prime. In order to simplify the analysis of
symmetries, the equation is changed to the real coordinates given by
$u=(x,y)\in\mathbb{R}^{2}$. In real coordinates, the equation is given by%
\begin{equation}
Lu+\partial_{s}^{2}g(u)=0, \label{EQL}%
\end{equation}
where $L$ is the linear operator
\begin{equation}
Lu:=-\nu^{2}\partial_{t}^{2}u-\partial_{s}^{4}u+a^{-2}R\partial_{s}%
^{2}u\text{,}%
\end{equation}
where
\[
R=diag(1,-1),
\]
and $g(u)=\mathcal{O}\left(  \left\vert u\right\vert ^{2}\right)  $ is
analytic for $\left\vert (x,y)\right\vert <1$.

We present some definitions and useful results about Sobolev spaces before
implementing the Lyapunov-Schmidt reduction. We use the inner product in the
space $L^{2}(T^{2};\mathbb{R}^{2})$ given by%
\[
\left\langle u_{1},u_{2}\right\rangle =\frac{1}{(2\pi)^{2}}\int_{T^{2}}%
u_{1}\cdot u_{2}~dt~ds\text{.}%
\]
Functions $u\in L^{2}(T^{2};\mathbb{R}^{2})$ have the Fourier representation%
\[
u=\sum_{(j,k)\in\mathbb{Z}^{2}}u_{j,k}e^{i(jt+ks)},\qquad u_{j,k}=\bar
{u}_{-j,-k}\in\mathbb{C}^{2}.
\]
The Sobolev space $H^{s}$ is the subspace of functions in $L^{2}$ with bounded
norm%
\[
\left\Vert u\right\Vert _{H^{s}}^{2}=\sum_{(j,k)\in\mathbb{Z}^{2}}\left\vert
u_{j,k}\right\vert ^{2}\left(  j^{2}+k^{2}+1\right)  ^{s}\text{.}%
\]
This space has the Banach algebra property for $s>1$,%
\[
\left\Vert uv\right\Vert _{H^{s}}\leq\left\Vert u\right\Vert _{H^{s}%
}\left\Vert v\right\Vert _{H^{s}}.
\]

The Banach algebra property implies that the nonlinear operator
$g(u)=\mathcal{O}(\left\Vert u\right\Vert _{H^{s}}^{2})$ is well defined and
continuous for $\left\Vert u\right\Vert _{H^{s}}<1$. The Lyapunov-Schmidt
reduction is implemented in the Sobolev space of functions with zero average,%
\[
H_{0}^{s}(T^{2};\mathbb{R}^{2})=\left\{  u\in H^{s}(T^{2};\mathbb{R}^{2}%
):\int_{T^{2}}u=0\right\}  .
\]

The linear operator $L:D(L)\rightarrow H_{0}^{s}$ is continuous when the
domain%
\[
D(L)=\{u\in H_{0}^{s}:Lu\in H_{0}^{s}\},
\]
is completed under the graph norm
\[
\left\Vert u\right\Vert _{L}^{2}=\left\Vert Lu\right\Vert _{H_{0}^{s}}%
^{2}+\left\Vert u\right\Vert _{H_{0}^{s}}^{2}~.
\]
In Fourier basis, the operator $L:D(L)\rightarrow H_{0}^{s}$ is given by
\[
Lu=\sum_{(j,k)\in\mathbb{Z}_{0}^{2}}\left(  \nu^{2}j^{2}I-k^{4}I+a^{-2}%
k^{2}R\right)  u_{j,k}e^{i(jt+ks)},
\]
where%
\[
\mathbb{Z}_{0}^{2}=\mathbb{Z}^{2}\backslash\{(0,0)\}\text{.}%
\]
Then, the eigenvalues of $L$ are
\begin{equation}
\lambda_{j,k,l}=\left(  \nu j\right)  ^{2}-k^{4}+\left(  -1\right)  ^{l}%
a^{-2}k^{2}\text{,}%
\end{equation}
for $(j,k,l)\in\mathbb{Z}_{0}^{2}\times\mathbb{Z}_{2}$. The set of
eigenfunctions of $L$, given by $e_{l}e^{i(jt+ks)}$ with
\[
e_{0}=(1,0)\text{ and }e_{1}=(0,1)\text{,}%
\]
is orthonormal and complete:
\[
Lu=\sum_{(j,k,l)\in\mathbb{Z}_{0}^{2}\times\mathbb{Z}_{2}}\lambda
_{j,k,l}\left\langle u,e_{l}e^{i(jt+ks)}\right\rangle e_{l}e^{i(jt+ks)}%
\text{.}%
\]

Choosing $a_{0}$ such that%
\begin{equation}
a_{0}^{-2}=\left(  -1\right)  ^{l_{0}}\left(  k_{0}^{2}-\left(  pj_{0}%
/qk_{0}\right)  ^{2}\right)  \label{a0}%
\end{equation}
for a fixed $\left(  j_{0},k_{0},l_{0}\right)  \in\mathbb{N}^{2}%
\times\mathbb{Z}_{2}$, we have $\lambda_{\pm j_{0},\pm k_{0},l_{0}}\left(
a_{0}\right)  =0$; the other eigenvalues satisfy%

\begin{equation}
\lambda_{j,k,l}(a_{0})=\left(  qj/p\right)  ^{2}-k^{4}+(-1)^{l+l_{0}}\left(
k_{0}^{2}-\left(  pj_{0}/qk_{0}\right)  ^{2}\right)  k^{2}\text{.}%
\end{equation}

\begin{definition}
Let $N\subset\mathbb{Z}_{0}^{2}\times\mathbb{Z}_{2}$ be the subset of all
lattice points corresponding to zero eigenvalues,
\[
N=\left\{  \left(  j,k,l\right)  \in\mathbb{Z}_{0}^{2}\times\mathbb{Z}%
_{2}:\lambda_{j,k,l}\left(  a_{0}\right)  =0\right\}  \text{.}%
\]

\end{definition}

By definition we have that the kernel of $L(a_{0})$ is generated by
eigenfunctions $e_{l}e^{i(jt+ks)}$ with $\left(  j,k,l\right)  \in N$. Notice
that additional sites to $(\pm j_{0},\pm k_{0},l_{0})$ may be present in $N$
due to resonances.

The Lyapunov-Schmidt reduction separates the kernel and the range equations
using the projections%
\[
Qu=\sum_{\left(  j,k,l\right)  \in N}u_{j,k,l}e_{l}e^{i(jt+ks)},\qquad
Pu=(I-Q)u.
\]
Setting
\[
u=v+w,\qquad v=Qu,\qquad w=Pu\text{,}%
\]
equation (\ref{EQL}) is equivalent to the kernel equation%
\begin{equation}
QLQv+Q\partial_{s}^{2}g(v+w)=0,
\end{equation}
and the range equation%
\begin{equation}
PLPw+P\partial_{s}^{2}g(v+w)=0.
\end{equation}

\begin{proof}
[Proof of Theorem \ref{1}]The proof is split in three propositions. In
Proposition \ref{2} we use the contraction mapping theorem to prove that the
range equation has a unique solution $w(v,a)\in H_{0}^{s}$ defined in a
neighborhood of $(0,a_{0})$, where $w=\mathcal{O}(\left\Vert v\right\Vert
_{H_{0}^{s}}^{2})$. Using this solution in the kernel equation we obtain the
bifurcation equation%
\begin{equation}
QLQv+Q\partial_{s}^{2}g(v+w(v,a))=0\text{,}%
\end{equation}
which is defined in a neighborhood of $(0,a_{0})\in\ker L(a_{0})\times
\mathbb{R}$.

Proposition \ref{4} proves that for each fixed positive $q$ there is an
infinite number of non-resonant amplitudes $a_{0}^{-2}\in\left(  0,2/q\right)
$ with $j_{0}=1$ and $l_{0}=0$. In Proposition \ref{3}, using the symmetries
and a non-resonant amplitude $a_{0}$, the bifurcation equation is reduced to a
subspace of dimension one within the kernel. Then, the existence of the local
bifurcation is obtained by the Crandall-Rabinowitz theorem, which gives the
estimates $v(t,a)=be_{l_{0}}\cos t\cos k_{0}s+\mathcal{O}(b^{2})$ and
$a=a_{0}+\mathcal{O}(b^{2})$ for $b\in\lbrack0,b_{0}\mathbb{]}$.

Estimates in Propositions \ref{2} and \ref{3} imply that%
\begin{equation}
w_{1}(t,s)=a+(v+w)(pt/q,s)=a_{0}+bi^{l_{0}}\cos(pt/q)\cos k_{0}s+\mathcal{O}%
_{H_{0}^{s}}(b^{2})\text{.} \label{sol}%
\end{equation}
The regularity of the solutions is obtained by the embedding $H_{0}^{s}\subset
C^{4}$ for $s\geq6$. The symmetries of $u=v+w$ follow from the symmetries of
$v$ in Propositions \ref{3}, i.e.%
\begin{align*}
u(t,s)  &  =u(-t,s)=u(t,-s)=u(t,s+2\pi/k_{0})\\
&  =Ru(t+l_{0}\pi,s)\text{.}%
\end{align*}
Finally, the symmetries of $w_{1}$ in the theorem follow from the symmetries
of $u$ after rescaling the period.
\end{proof}

\section{The range equation}

In this section, the range equation is solved as a fixed point $w(a,v)\in
H_{0}^{s}$ of the operator%
\[
Kw=-\left(  PLP\right)  ^{-1}\partial_{s}^{2}g(w+v,a)\text{.}%
\]
The key element in the proof consists in showing that
\[
\left(  PLP\right)  ^{-1}\partial_{s}^{2}:H^{s}\rightarrow H_{0}^{s}%
\]
is well defined and bounded. Once this result is established, the solution is
obtained by an application of the contraction mapping theorem to the nonlinear
operator
\[
Kw=\mathcal{O}(\varepsilon^{-1}\left\Vert w\right\Vert _{H_{0}^{s}}%
^{2}):B_{\rho}\subset H_{0}^{s}\rightarrow H_{0}^{s}.
\]

\begin{lemma}
Assume that $2\varepsilon<a_{0}^{-2}/2<1/q-2\varepsilon$ and $\left\vert
a^{-2}-a_{0}^{-2}\right\vert \lesssim\varepsilon$. Then, we have the estimate
\begin{equation}
\left\vert \lambda_{j,k,l}(a)\right\vert \gtrsim\varepsilon\left(
k^{2}+\left\vert j\right\vert \right)  \text{ for }\left(  j,k,l\right)  \in
N^{c}\text{.}%
\end{equation}

\end{lemma}

\begin{proof}
The inequality $\left\vert pj/q-k^{2}\right\vert \geq1/q$ is true unless
$pj/q=k^{2}$. In the case that $pj/q=k^{2}$, then $\lambda_{j,k,l}(a_{0})=\pm
a_{0}^{-2}k^{2}$ and%
\[
\left\vert \lambda_{j,k,l}(a_{0})\right\vert \gtrsim2\varepsilon k^{2}%
\gtrsim2\varepsilon\left(  k^{2}+\left\vert j\right\vert \right)  .
\]
We may assume that $j\geq0$, since the case $j\leq0$ follows by analogy. For
the case $\left\vert pj/q-k^{2}\right\vert \geq1/q$ and $j\geq0$, we have
\[
\left\vert \lambda_{j,k,l}(a_{0})\right\vert =\left\vert pj/q+\mu
_{k}\right\vert \left\vert pj/q-\mu_{k}\right\vert ,
\]
where $\mu_{k}=\sqrt{k^{4}\pm a_{0}^{-2}k^{2}}$. Since $\lim_{k\rightarrow
\infty}\left\vert k^{2}-\mu_{k}\right\vert =a_{0}^{-2}/2$, then%
\[
\left\vert pj/q-\mu_{k}\right\vert \geq\left\vert pj/q-k^{2}\right\vert
-\left\vert \mu_{k}-k^{2}\right\vert \geq\frac{1}{q}-\left(  \frac{1}{2}%
a_{0}^{-2}+\varepsilon\right)  >\varepsilon~,
\]
for $\left\vert j\right\vert +\left\vert k\right\vert \geq M$ with $M$ big.
Therefore, we have the estimate
\[
\left\vert \lambda_{j,k,l}(a_{0})\right\vert \geq\varepsilon\left\vert
pj/q+\mu_{k}\right\vert \geq c\varepsilon\left(  k^{2}+\left\vert j\right\vert
\right)
\]

We can adjust the constant $\varepsilon$ such that the estimate
\[
\left\vert \lambda_{j,k,l}(a)\right\vert \geq c\varepsilon\left(
k^{2}+\left\vert j\right\vert \right)
\]
is true for all $(j,k,l)\in N^{c}$. Since $\left\vert a^{-2}-a_{0}%
^{-2}\right\vert <c\varepsilon$ and%
\[
L(a)=L(a_{0})\pm\left(  a^{-2}-a_{0}^{-2}\right)  \partial_{s}^{2}~\text{,}%
\]
we conclude that
\[
\left\vert \lambda_{j,k,l}(a)\right\vert \geq\left\vert \lambda_{j,k,l}%
(a_{0})\right\vert -c\varepsilon k^{2}\geq c\varepsilon\left(  k^{2}%
+\left\vert j\right\vert \right)  \text{.}%
\]

\end{proof}

\begin{lemma}
Assume that $2\varepsilon<a_{0}^{-2}/2<1/q-2\varepsilon$ and $\left\vert
a^{-2}-a_{0}^{-2}\right\vert \lesssim\varepsilon$. The linear operator
$\left(  PLP\right)  ^{-1}\partial_{s}^{2}$ is continuous with%
\begin{equation}
\left\Vert \left(  PLP\right)  ^{-1}\partial_{s}^{2}w\right\Vert _{H_{0}^{s}%
}\lesssim\varepsilon^{-1}\left\Vert w\right\Vert _{H_{0}^{s}}\text{.}%
\end{equation}

\end{lemma}

\begin{proof}
By the previous lemma $\left\vert \lambda_{j,k,l}(a)\right\vert \gtrsim
\varepsilon k^{2}$ for $(j,k,l)\in N^{c}$. Then, the estimate
\[
\left\Vert \left(  PLP\right)  w\right\Vert _{H_{0}^{s}}\gtrsim\varepsilon
\left\Vert \partial_{s}^{2}Pw\right\Vert _{H_{0}^{s}}%
\]
holds true with $w\in H_{0}^{s}$. Applying this estimate to $\left(
PLP\right)  ^{-1}w\in D(L)\subset H_{0}^{s}$, we obtain%
\[
\left\Vert \partial_{s}^{2}\left(  PLP\right)  ^{-1}w\right\Vert _{H_{0}^{s}%
}\lesssim\varepsilon^{-1}\left\Vert Pw\right\Vert _{H_{0}^{s}}\text{.}%
\]
Since $H_{0}^{s}\subset C^{4}$ for $s\geq6$, then $\partial_{s}^{2}$ and
$\left(  PLP\right)  ^{-1}$ commute. Therefore, the operator $\left(
PLP\right)  ^{-1}\partial_{s}^{2}:PH^{s}\rightarrow PH_{0}^{s}\ $is well
define and bounded by $\mathcal{O}\left(  \varepsilon^{-1}\right)  $.
\end{proof}

\begin{proposition}
\label{2}Assume $a_{0}^{-2}\in\left(  0,2/q\right)  $. There is a unique
continuous solution $w(v,a)\in H_{0}^{s}$ of the range equation defined for
$(v,a)$ in a small neighborhood of $(0,a_{0})\in\ker L(a_{0})\times\mathbb{R}$
such that
\begin{equation}
\left\Vert w(v,a)\right\Vert _{H_{0}^{s}}\lesssim\varepsilon^{-1}\left\Vert
v\right\Vert ^{2}\text{,}%
\end{equation}
for small $\varepsilon$.
\end{proposition}

\begin{proof}
By the Banach algebra property of $H^{s}$, the operator
\[
g(w)=\mathcal{O}(\left\Vert w\right\Vert _{H_{0}^{s}}^{2}):B_{\rho}\rightarrow
H^{s}%
\]
is well define in the domain $B_{\rho}=\{w\in H_{0}^{s}:\left\Vert
w\right\Vert _{H_{0}^{s}}<\rho\}$ for $\rho<1$. Since $a_{0}^{-2}\in\left(
0,2/q\right)  $, we can chose a small enough $\varepsilon$ such that the
hypothesis of the previous lemma hold true. Therefore,
\begin{align*}
Kw  &  =-\partial_{s}^{2}\left(  PLP\right)  ^{-1}g(w+v,a)=\mathcal{O}%
(\varepsilon^{-1}\left\Vert w\right\Vert _{H_{0}^{s}}^{2})\\
&  :B_{\rho}\subset PH_{0}^{s}\rightarrow PH_{0}^{s}\text{,}%
\end{align*}
is well defined and continuous. Moreover, it is a contraction for $\rho$ of
order $\rho=\mathcal{O}(\varepsilon)$. By the contraction mapping theorem,
there is a unique continuous fixed point $w(v,a)\in B_{\rho}$. The estimate
$\left\Vert w(v,a)\right\Vert _{H_{0}^{s}}\leq\varepsilon^{-1}\left\Vert
v\right\Vert ^{2}$ is obtained from
\[
\left\Vert Kw\right\Vert _{H_{0}^{s}}\lesssim\varepsilon^{-1}\left(
\left\Vert w\right\Vert _{H_{0}^{s}}^{2}+\left\Vert v\right\Vert ^{2}\right)
.
\]

\end{proof}

\begin{remark}
Since $\lambda_{j,k}\geq\varepsilon\left(  k^{2}+j\right)  $, the domain
$D(L)$ is compactly contained in $H_{0}^{s}$. However, we cannot prove the
global bifurcation by the classical Rabinowitz theorem because $\left(
PLP\right)  ^{-1}\partial_{s}^{2}g(u)$ is not compact, but only continuous.
This lack of compactness is the reason why we cannot obtain the regularity by
bootstrapping arguments. Instead, the regularity of the solutions is obtained
using the Sobolev embedding $H_{0}^{s}\subset C^{4}$ for $s\geq6$.
\end{remark}

\section{The bifurcation equation}

In this section, the bifurcation equation is solved by an application of the
Crandall-Rabinowitz theorem to the case of non-resonant $a_{0}$'s with
$a_{0}^{-2}\in\left(  0,2/q\right)  $.

\begin{definition}
\label{NR}An $a_{0}$ is \emph{non-resonant }for the lattice point
$(j_{0},k_{0},l_{0})\in\mathbb{N}^{2}\times\mathbb{Z}_{2}$ if
\begin{equation}
N\cap\left(  j_{0}\mathbb{Z}\times k_{0}\mathbb{Z\times Z}_{2}\right)
=\{\left(  \pm j_{0},\pm k_{0},l_{0}\right)  \}.
\end{equation}

\end{definition}

\begin{proposition}
\label{4}For each $q$, there is an infinite number of non-resonant $a_{0}$'s
for $k_{0}\in\mathbb{N}$, $j_{0}=1$ and $l_{0}=0$, given by
\begin{equation}
a_{0}^{-2}=\frac{2}{q}-\frac{1}{k_{0}^{2}q^{2}},\text{\qquad}p=qk_{0}^{2}-1.
\end{equation}

\end{proposition}

\begin{proof}
First we fix positive numbers $p$ and $q$. The condition
\[
a_{0}^{-2}=\left(  -1\right)  ^{l}\left(  k^{2}-\left(  pj/qk\right)
^{2}\right)  \in\left(  0,2/q\right)
\]
holds for the infinite number of lattice points $\left(  j,k,l\right)  $ with
$jp=qk^{2}-1$ and $l=0$. Then
\[
a_{0}^{-2}=k^{2}-\left(  k-\frac{1}{qk}\right)  ^{2}=\frac{2}{q}-\frac
{1}{k^{2}q^{2}}.
\]
$\allowbreak$

By Proposition (\ref{2}), there is a finite number of elements $\left(
j_{m},k_{m},l_{m}\right)  $ corresponding to a non-resonant amplitude $a_{0}$.
That is,%
\[
a_{0}^{-2}=\left(  -1\right)  ^{l_{m}}\left(  k_{m}^{2}-\left(  pj_{m}%
/qk_{m}\right)  ^{2}\right)
\]
for $m\in\{0,...,M\}$. Therefore, there is an infinite number of $a_{0}%
^{-2}\in\left(  0,2/q\right)  $ with a finite number of resonances.We say that
$(j_{0},k_{0})$ is a maximal lattice point if $j_{m}<j_{0}$ or $k_{m}<k_{0}$
when $j_{m}=j_{0}$ for $m\neq0$. Let $(j_{0},k_{0})$ be a maximal lattice
point such that
\[
a_{0}^{-2}=\left(  -1\right)  ^{l_{0}}\left(  k_{0}^{2}-\left(  pj_{0}%
/qk_{0}\right)  ^{2}\right)  ,
\]
then one has that
\[
N\cap\left(  j_{0}\mathbb{Z}\times k_{0}\mathbb{Z\times Z}_{2}\right)  =\{(\pm
j_{0},\pm k_{0},l_{0})\}\text{.}%
\]
Therefore, there is an infinite number of non-resonant $a_{0}$'s with
$a_{0}^{-2}\in\left(  0,2/q\right)  $.

The choice of a maximal $j_{0}$ is equivalent to choose a maximal
$p_{0}=pj_{0}$. That is, we have $a_{0}^{-2}=\left(  -1\right)  ^{l_{0}%
}\left(  k_{0}^{2}-\left(  p_{0}/qk_{0}\right)  ^{2}\right)  $ for the numbers
$p_{0}$ and $q$ and
\[
N\cap\left(  \mathbb{Z}\times k_{0}\mathbb{Z\times Z}_{2}\right)  =\{(\pm1,\pm
k_{0},l_{0})\}.
\]
Therefore, for each fixed $q$, and possibly different numbers $p_{0}$, there
is an infinite number of non-resonant amplitudes $a_{0}^{-2}\in\left(
0,2/q\right)  $ with $j_{0}=1$ and $l_{0}=0$.
\end{proof}

\begin{remark}
The choice of maximal $p_{0}$ leads to the choice of a minimal period $2\pi
q/p_{0}$ for the bifurcation. This argument is similar to the argument used in
\cite{HKiel} for the wave equation.
\end{remark}

To apply the Crandall-Rabinowitz theorem we need to reduce the bifurcation
equation to a subspace of dimension one. This is attained by exploiting the
equivariance of the problem. The equation is equivariant under the action of
the group $G=\mathbb{Z}_{2}\times O(2)\times O(2)$ given by
\[
\rho(\tau,\sigma)u(t,s)=u(t+\tau,s+\sigma)~,
\]
for the abelian part, and%
\[
\rho(\kappa_{1})u(t,s)=u(-t,s),\quad\rho(\kappa_{2})u(t,s)=u(t,-s),\quad
\rho(\kappa_{3})u(t,s)=Ru(t,s),
\]
for the reflections. By the uniqueness of $w(v,a)$, the bifurcation equation
has the same equivariant properties that the differential equation. This
property is used in the following proposition to reduce the bifurcation
equation to a subspace of dimension one.

\begin{proposition}
\label{3}Let $a_{0}^{-2}\in\left(  0,2/q\right)  $ be a non-resonant amplitude
for the lattice point\emph{ }$(1,k_{0},l_{0})\in\mathbb{N}^{2}\mathbb{\times
Z}_{2}$. The bifurcation equation has a local continuum of $2\pi q/p$-periodic
solution bifurcating from the initial configuration with amplitude $a_{0}$.
These solutions satisfy the estimates%
\begin{equation}
v(t,s)=be_{l_{0}}\cos t\cos k_{0}s+\mathcal{O}(b^{2})\text{,\qquad}%
a=a_{0}+\mathcal{O}(b^{2}),
\end{equation}
and symmetries%
\begin{equation}
v(t,s)=v(-t,s)=v(t,-s)=v(t,s+2\pi/k_{0})=Rv(t+l_{0}\pi,s)\text{.}%
\end{equation}

\end{proposition}

\begin{proof}
In the Fourier basis, the action of $G$ is given by%
\[
\rho(\varphi)u_{j,k}=e^{ij\varphi}u_{j,k},\qquad\rho(\theta)u_{j,k}%
=e^{ik\theta}u_{j,k},
\]
for the abelian part and%
\[
\rho(\kappa_{1})u_{j,k}=u_{-j,k},\qquad\rho(\kappa_{2})u_{j,k}=u_{j,-k}%
,\qquad\rho(\kappa_{3})u_{j,k}=Ru_{j,k},
\]
for the reflections. Setting $u_{j,k}=\left(  u_{j,k,0},u_{j,k,1}\right)  $,
the irreducible representations correspond to the subspaces generated by
$(u_{j,k,l},u_{j,-k,l})\in\mathbb{C}^{2}$. Indeed, the linear operator $L$ has
blocks $\lambda_{j,k,l}I$ in these irreducible representations, which is
predicted by Schur's lemma.

Set the irreducible representation
\[
(u_{1},u_{2})=(u_{1,k_{0},l_{0}},u_{1,-k_{0},l_{0}})\text{.}%
\]
The action of the group in this representation is%
\[
\rho(\varphi)(u_{1},u_{2})=e^{i\varphi}(u_{1},u_{2}),\qquad\rho(\theta
)(u_{1},u_{2})=(e^{ik_{0}\theta}u_{1},e^{-ik_{0}\theta}u_{2}),
\]
and%
\[
\rho(\kappa_{1})(u_{1},u_{2})=(\bar{u}_{2},\bar{u}_{1}),~~\rho(\kappa
_{2})(u_{1},u_{2})=(u_{2},u_{1}),~~\rho(\kappa_{3})(u_{1},u_{2})=\left(
-1\right)  ^{l_{0}}(u_{1},u_{2}).
\]
Therefore, the group%
\[
S=\left\langle \kappa_{1},\kappa_{2},\left(  l_{0}\pi,\kappa_{3}\right)
,(\pi,\pi/k_{0})\right\rangle
\]
has fixed point space $(u_{1},u_{2})=(b,b)$ for $b\in\mathbb{R}$ in this representation.

Set
\[
\ker L^{S}(a_{0}):=\ker L(a_{0})\cap\emph{Fix~}(S)\text{.}%
\]
The bifurcation equation%
\begin{equation}
QLQw+Q\partial_{s}^{2}g(v+w(v,a)):\ker L^{S}(a_{0})\times\mathbb{R}%
\rightarrow\ker L^{S}(a_{0}) \label{BE}%
\end{equation}
is well defined by the equivariant properties. Since for a non-resonant
amplitude $a_{0}$, the kernel consist of the subspace $(u_{1},u_{2})=(b,b)$
for $b\in\mathbb{R}$, then the kernel in the fixed point space of $S$ is
generated by the simple eigenfunction%
\[
\sum_{\left(  j,k,l\right)  \in N}e_{l}e^{i(jt+ks)}=4e_{l_{0}}\cos j_{0}t\cos
k_{0}s\text{.}%
\]
Therefore,%
\[
\ker L^{S}(a_{0})=\{be_{l_{0}}\cos j_{0}t\cos k_{0}s:b\in\mathbb{R}\}\text{.}%
\]

Since $\ker L^{S}(a_{0})$ has dimension one, the local bifurcation for $a$
close to $a_{0}$ follows from the Crandall-Rabinowitz theorem applied to the
bifurcation equation (\ref{BE}). It is only necessary to verify that
$\partial_{a}L(a)\left(  e_{l_{0}}\cos j_{0}t\cos k_{0}s\right)  $ is not in
the range of $L$. This follows from%
\[
\partial_{a}L(a)\left(  e_{l_{0}}\cos j_{0}t\cos k_{0}s\right)  =-2a^{-3}%
k_{0}^{2}Re_{l_{0}}\left(  \cos j_{0}t\cos k_{0}s\right)  \in\ker L(a_{0}).
\]

The estimates $a=a_{0}+\mathcal{O}(b)$ and
\[
v(t,s)=be_{l_{0}}\cos j_{0}t\cos k_{0}s+\mathcal{O}(b^{2})
\]
are consequence of the Crandall-Rabinowitz theorem. Moreover, the $S^{1}%
$-action of the element $\varphi=\pi/j_{0}$ in the kernel generated by
$e_{l_{0}}\cos j_{0}t\cos k_{0}s$ is given by $\rho(\varphi)=-1$. This
symmetry implies that the bifurcation equation is odd and $a=a_{0}%
+\mathcal{O}(b^{2})$.
\end{proof}

\section{Traveling waves}

The irreducible representation $(u_{1,k_{0},l_{0}},u_{1,-k_{0},l_{0}})$ has
another isotropy group given by%
\[
T=\left\langle \kappa_{1}\kappa_{2},\left(  l_{0}\pi,\kappa_{3}\right)
,(\varphi,-\varphi/k_{0})\right\rangle .
\]
This isotropy group has a one dimensional fixed point space corresponding to
$(u_{1},u_{2})=(b,0)$ for $b\in\mathbb{R}$. Solutions with isotropy group $T$
are traveling waves of the form $u(\nu t+s)$ for $k_{0}=1$.

For these traveling waves, the PDE becomes the ODE%
\begin{equation}
-u^{\prime\prime}-\nu^{2}u+a^{-2}Ru+g(u)=0\text{.}%
\end{equation}
The spectrum of the linear operator associated to the bifurcation problem is%
\[
\lambda_{j}=j^{2}-\nu^{2}+(-1)^{l}a^{-2}\text{.}%
\]
Actually, the global bifurcation of traveling waves for filaments has been
proven in \cite{GaIz12} applying equivariant degree theory to the reduced ODE.
In a similar manner, one can prove the following theorem.

\begin{theorem}
The equation (\ref{EQ}) has a \emph{global bifurcation} of traveling waves
starting from the initial configuration $u=a$ with frequency
\[
\nu_{0}=\sqrt{1+(-1)^{l}a^{-2}}\in\mathbb{R}^{+}\text{.}%
\]
The local bifurcation can be parameterized by $b$ with the estimate
$\nu(b)=\nu_{0}+\mathcal{O}(b^{2})$ and%
\[
u(\nu t+s)=a+bi^{l}\cos\left(  \nu t+s\right)  +\mathcal{O}_{C^{4}}%
(b^{2})\text{.}%
\]

\end{theorem}

Observe that the set of traveling waves forms a two-dimensional family
parameterized by amplitude $a$ and frequency $\nu$, while standing waves exist
for an infinite number of local and continuous curves that are parameterized
by amplitude $a$ and have fixed rational frequency $\nu$.

\noindent\textbf{Acknowledgement.} The author is grateful to W. Craig and H.
Kielh\"{o}fer for useful discussions related to this project. This project is
supported by PAPIIT-UNAM grant IA105217.

\end{document}